\documentclass[11pt]{amsart}
\usepackage{amsopn}
\usepackage{amssymb, amscd}
\usepackage{graphicx}

\newcommand{\nc}{\newcommand}

\nc{\eg}{\mathfrak{e} } \nc{\fg}{\mathfrak{f} }
\nc{\vg}{\mathfrak{v} } \nc{\wg}{\mathfrak{w} }
\nc{\zg}{\mathfrak{z} } \nc{\ngo}{\mathfrak{n} }
\nc{\kg}{\mathfrak{k} } \nc{\mg}{\mathfrak{m} }
\nc{\bg}{\mathfrak{b} } \nc{\ggo}{\mathfrak{g} }
\nc{\ggob}{\overline{\mathfrak{g}} } \nc{\sog}{\mathfrak{so} }
\nc{\sug}{\mathfrak{su} } \nc{\spg}{\mathfrak{sp} }
\nc{\slg}{\mathfrak{sl} } \nc{\glg}{\mathfrak{gl} }
\nc{\cg}{\mathfrak{c} } \nc{\rg}{\mathfrak{r} }
\nc{\hg}{\mathfrak{h} } \nc{\tg}{\mathfrak{t} }
\nc{\ug}{\mathfrak{u} } \nc{\dg}{\mathfrak{d} }
\nc{\ag}{\mathfrak{a} } \nc{\pg}{\mathfrak{p} }
\nc{\sg}{\mathfrak{s} } \nc{\affg}{\mathfrak{aff} }

\nc{\pca}{\mathcal{P}} \nc{\nca}{\mathcal{N}}
\nc{\lca}{\mathcal{L}} \nc{\oca}{\mathcal{O}}
\nc{\mca}{\mathcal{M}} \nc{\tca}{\mathcal{T}}
\nc{\aca}{\mathcal{A}} \nc{\cca}{\mathcal{C}}
\nc{\gca}{\mathcal{G}} \nc{\sca}{\mathcal{S}}
\nc{\hca}{\mathcal{H}} \nc{\bca}{\mathcal{B}}
\nc{\dca}{\mathcal{D}} \nc{\val}{\operatorname{val}}

\nc{\vp}{\varphi} \nc{\ddt}{\tfrac{{\rm d}}{{\rm d}t}}
\nc{\dpar}{\tfrac{\partial}{\partial t}} \nc{\im}{\mathtt{i}}
\renewcommand{\Re}{{\rm Re}}

\nc{\SO}{\mathrm{SO}} \nc{\Spe}{\mathrm{Sp}} \nc{\Sl}{\mathrm{SL}}
\nc{\SU}{\mathrm{SU}} \nc{\Or}{\mathrm{O}} \nc{\U}{\mathrm{U}}
\nc{\Gl}{\mathrm{GL}} \nc{\Se}{\mathrm{S}} \nc{\Cl}{\mathrm{Cl}}
\nc{\Spein}{\mathrm{Spin}} \nc{\Pin}{\mathrm{Pin}}
\nc{\G}{\mathrm{GL}_n(\RR)} \nc{\g}{\mathfrak{gl}_n(\RR)}

\nc{\RR}{{\Bbb R}} \nc{\HH}{{\Bbb H}} \nc{\CC}{{\Bbb C}}
\nc{\ZZ}{{\Bbb Z}} \nc{\FF}{{\Bbb F}} \nc{\NN}{{\Bbb N}}
\nc{\QQ}{{\Bbb Q}} \nc{\PP}{{\Bbb P}}

\nc{\vs}{\vspace{.2cm}} \nc{\vsp}{\vspace{1cm}}
\nc{\ip}{\langle\cdot,\cdot\rangle} \nc{\ipp}{(\cdot,\cdot)}
\nc{\la}{\langle} \nc{\ra}{\rangle} \nc{\unm}{\tfrac{1}{2}}
\nc{\unc}{\tfrac{1}{4}} \nc{\und}{\tfrac{1}{16}}
\nc{\no}{\vs\noindent} \nc{\lamn}{\Lambda^2(\RR^n)^*\otimes\RR^n}
\nc{\lamp}{\Lambda^2\pg^*\otimes\pg}
\nc{\lamg}{\Lambda^2\ggo^*\otimes\ggo}
\nc{\lamngo}{\Lambda^2\ngo^*\otimes\ngo} \nc{\tangz}{{\rm T}^{\rm
Zar}} \nc{\mum}{/\!\!/} \nc{\kir}{/\!\!/\!\!/}
\nc{\Ri}{\tfrac{4\Ric_{\mu}}{||\mu||^2}} \nc{\ds}{\displaystyle}
\nc{\ben}{\begin{enumerate}} \nc{\een}{\end{enumerate}}
\nc{\f}{\frac} \nc{\lb}{[\cdot,\cdot]}
\nc{\isn}{\tfrac{1}{||v||^2}} \nc{\gkp}{(\ggo=\kg\oplus\pg,\ip)}
\nc{\ukh}{(\ug=\kg\oplus\hg,\ip)}

\nc{\Hess}{\operatorname{Hess}} \nc{\ad}{\operatorname{ad}}
\nc{\Ad}{\operatorname{Ad}} \nc{\rank}{\operatorname{rank}}
\nc{\Irr}{\operatorname{Irr}} \nc{\End}{\operatorname{End}}
\nc{\Aut}{\operatorname{Aut}} \nc{\Inn}{\operatorname{Inn}}
\nc{\Der}{\operatorname{Der}} \nc{\Ker}{\operatorname{Ker}}
\nc{\Iso}{\operatorname{I}} \nc{\Diff}{\operatorname{Diff}}
\nc{\Lie}{\operatorname{L}} \nc{\tr}{\operatorname{tr}}
\nc{\dif}{\operatorname{d}} \nc{\sen}{\operatorname{sen}}
\nc{\modu}{\operatorname{mod}} \nc{\Riem}{\operatorname{Rm}}
\nc{\Ricci}{\operatorname{Ric}} \nc{\sym}{\operatorname{sym}}
\nc{\symac}{\operatorname{sym^{ac}}}
\nc{\symc}{\operatorname{sym^{c}}} \nc{\scalar}{\operatorname{sc}}
\nc{\grad}{\operatorname{grad}} \nc{\ricci}{\operatorname{Ric}}
\nc{\nr}{\operatorname{nr}} \nc{\riccic}{\operatorname{ric^{c}}}
\nc{\riccig}{\operatorname{ric^{\gamma}}}
\nc{\Rin}{\operatorname{M}} \nc{\Le}{\operatorname{L}}
\nc{\tang}{\operatorname{T}} \nc{\level}{\operatorname{level}}
\nc{\rad}{\operatorname{r}} \nc{\abel}{\operatorname{ab}}
\nc{\CH}{\operatorname{CH}} \nc{\mcc}{\operatorname{mcc}}
\nc{\Adj}{\operatorname{Adj}} \nc{\Order}{\operatorname{O}}
\nc{\inj}{\operatorname{inj}}\nc{\R}{\operatorname{R}}
\nc{\Spec}{\operatorname{Spec}} \nc{\I}{\operatorname{I}}
\nc{\ric}{\operatorname{Rc}}

\theoremstyle{plain}
\newtheorem{theorem}{Theorem}[section]
\newtheorem{proposition}[theorem]{Proposition}
\newtheorem{corollary}[theorem]{Corollary}
\newtheorem{lemma}[theorem]{Lemma}

\theoremstyle{definition}
\newtheorem{definition}[theorem]{Definition}

\theoremstyle{remark}
\newtheorem{remark}[theorem]{Remark}

\newtheorem{example}[theorem]{Example}

\title{The Ricci flow in a class of solvmanifolds.}
\author{Romina M. Arroyo}
\address{FaMAF y CIEM, Universidad Nacional de C\'ordoba, C\'ordoba, Argentina}
\email{arroyo@famaf.unc.edu.ar}
\thanks{This research was supported by a fellowship from CONICET and grants from CONICET, FONCYT and SeCyT (Universidad Nacional de C\'ordoba)}
\begin{document}
\maketitle
\begin{abstract}
In this paper, we study the Ricci flow of solvmanifolds whose Lie algebra has an abelian ideal of codimension one, by using the bracket flow. We prove that solutions to the Ricci flow are immortal, the $\omega$-limit of bracket flow solutions is a single point, and that for any sequence of times there exists a subsequence in which the Ricci flow converges, in the pointed topology, to a manifold which is locally isometric to a flat manifold. We give a functional which is non-increasing along a normalized bracket flow that will allow us to prove that given a sequence of times, one can extract a subsequence converging to an algebraic soliton, and to determine which of these limits are flat. Finally, we use these results to prove that if a Lie group in this class admits a Riemannian metric of negative sectional curvature, then the curvature of any Ricci flow solution will become negative in finite time.
\end{abstract}
\section{Introduction}\label{int}
The Ricci flow is an evolution equation for a curve of Riemannian metrics on a manifold. In recent years, the Ricci flow has proven to be a very important tool. Many strong results, not only in Riemannian geometry, have been proven by using this equation. The objective of this paper is to study the Ricci flow for solvmanifolds whose Lie algebra has an abelian ideal of codimension one and get similar results to those obtained by J. Lauret in \cite{riccinil} in the case of nilmanifolds.

Let $(G,g)$ be a solvmanifold, i.e. a simply connected solvable Lie group $G$ endowed with a left-invariant metric $g.$ Assume that the Lie algebra of $G$ has an abelian ideal of codimension one. Consider the Ricci flow starting at $g$, that is,
\[
\dpar g(t) = -2\ric(g(t)), \quad g(0)=g.
\]
The solution $g(t)$ is a left-invariant metric for all $t$, thus each $g(t)$ is determined by an inner product on the Lie algebra. We will follow the approach in \cite{riccifl} to study the evolution of these metrics by varying Lie brackets instead of inner products.

More precisely, let $\mu$ be a Lie bracket on $\RR^{n+1}$ with an abelian ideal of codimension one. We may assume that $\mu$ is determined by $A=\ad_{\mu}(e_0)|_{\RR^n} \in \g,$ where $\RR^{n+1} = \RR e_0 \oplus \RR^n$ and $\RR^n$ is the abelian ideal, and so it will be denoted by $\mu_A.$ Each $\mu_A$ determines a Riemannian manifold $(G_{\mu_A},g_{\mu_A}),$ where $G_{\mu_A}$ is the simply connected Lie group with Lie algebra $(\RR^{n+1},\mu_A)$ and $g_{\mu_A}$ is the left-invariant metric determined by $\ip,$ the canonical inner product on $\RR^{n+1}.$ Every solvmanifold whose Lie algebra has an abelian ideal of codimension one is isometric to some $\mu_A$ (see Section \ref{pre}). By \cite[Theorem 3.3]{riccifl}, the Ricci flow solution is given by $g(t) = \varphi(t)^{*}g_{\mu(t)}$, where $\mu(t)$ is a family of Lie brackets solving a ODE called the bracket flow, and $\varphi(t):G \rightarrow G_{\mu(t)}$ is the Lie group isomorphism with derivative $h(t): (\RR^{n+1},\mu) \rightarrow (\RR^{n+1},\mu(t)),$ and $h(t)$ satisfies
$$
\tfrac{d}{dt} h = -h \ricci (\la \cdot, \cdot \ra_t), \quad \tfrac{d}{dt} h = - \ricci_{\mu(t)} h, \quad h(0)=I.
$$ In our case, we see that $\mu(t)=\mu_{A(t)},$ where $A(t) \in \g$ is the solution to the following ODE,
$$
{\tfrac{d}{dt}} A=-\tr(S(A)^{2})A+\tfrac{1}{2}[A,[A,A^{t}]]-\tfrac{1}{2}\tr(A)[A,A^{t}], \quad A(0)=A,
$$ and then we study the evolution of the matrix $A$. The main results in this paper can be summarized as follows:
\begin{itemize}
\item The Ricci flow solution $g(t)$ is defined for all $t \in (T_-,\infty),$ where $-\infty < T_-< 0,$ and if $\tr ({A}^2) \geq 0,$ then $g(t)$ is a Type-III solution (see Proposition \ref{a} and Proposition \ref{type3}).
\item The scaling-invariant functional $\tfrac{\|[A(t),A(t)^t]\|}{\|A(t)\|^2}$ is strictly decreasing unless $\mu_{A}$ is an algebraic soliton, in which case it is constant (see Lemma \ref{func}). This happens precisely when $A$ is either normal or nilpotent of a special kind (see Proposition \ref{solmuA}).
\item For any sequence $t_k \rightarrow \infty,$ there exists a subsequence of $(G_{\mu_{A(t_k)}}, g_{\mu_{A(t_k)}})$ which converges in the pointed topology to a flat manifold, up to local isometry (see Corollary \ref{subconv}).
\item If $\tr(A)=0$ (i.e. $G_{\mu_{A}}$ unimodular), then $B(t)=\tfrac{A(t)}{\|A(t)\|}$ converges to a matrix $B_{\infty},$ as $t \rightarrow \infty$ (see Lemma \ref{cociente} and Remark \ref{nochaos}).
\item For any sequence $t_k \rightarrow \infty,$ there exists a subsequence of $(G_{\mu_{B(t_k)}}, g_{\mu_{B(t_k)}})$ which converges in the pointed topology to $(G_{\mu_{B_{\infty}}}, g_{\mu_{B_{\infty}}})$ (up to local isometry), which is an algebraic soliton. In addition, $(G_{\mu_{B_{\infty}}}, g_{\mu_{B_{\infty}}})$ is non-flat, unless every eigenvalue of $A$ is purely imaginary (see Theorem \ref{convnorsol}).
\item If $G_{\mu_{A}}$ admits a negatively curved left-invariant metric, then there exists $t_0 > 0$ such that $g(t)$ is negatively curved for all $t \geq t_0$ (see Theorem \ref{to2}). This is not true in general for solvmanifolds (see Example \ref{ejsol}).
\end{itemize}
\vs \noindent{\it Acknowledgements.} I wish to express my deep
gratitude to my advisor, Jorge Lauret, for his invaluable guidance.
I am also grateful to Ramiro Lafuente and Roberto Miatello for helpful observations.
\section{Preliminaries}\label{pre}
\subsection{The Ricci flow}\label{frvh}
Let $(M,g)$ be a Riemannian manifold. The Ricci flow starting at $(M,g)$ is the following partial differential equation:
\begin{equation}\label{fr}
\dpar g(t) = -2\ric(g(t)), \quad g(0)=g,
\end{equation} where $g(t)$ is a curve of Riemannian metrics on $M$ and
$\ric(g(t))$ the Ricci tensor of the metric $g(t)$.

A complete Riemannian metric $g$ on a differentiable manifold $M$
is a Ricci soliton if its Ricci tensor satisfies
$$
\ric(g)=cg + L_X g, \mbox{ for some $c \in \mathbb{R}$, $X \in
\chi (M)$ complete},
$$ where $\chi(M)$ denotes the space of differentiable vector
fields on $M$ and $L_X$ the usual Lie derivative in the
direction of the field $X.$

Equivalently, Ricci solitons are precisely the metrics that evolve along the Ricci flow only by the action of diffeomorphisms and scaling (i.e. $g(t) = c(t) \varphi(t)^{*}g $), giving geometries that are equivalent to the starting point, for all time $t$ (see \cite{libro} for more information about Ricci solitons).

\begin{definition}\label{type3}
A Ricci flow solution $g(t)$ is said to be of Type-III if it is defined for $t \in
[0,\infty)$ and there exists $C \in \RR$ such that
$$
\|\Riem (g(t))\| \leq \tfrac{C}{t}, \qquad \forall t \in (0,\infty),
$$ where $\Riem(g(t))$ is the Riemann curvature tensor of the metric $g(t).$
\end{definition}
\subsection{Varying Lie brackets.}\label{vlb}
We fix $(\RR^n,\ip),$ with $\ip$ an inner product on $\RR^n$ and we define
$$
\mathfrak{L}_n = \{\mu: \RR^n \times \RR^n\rightarrow \RR^n: \mu {\hspace{.2cm}} \mbox{ is bilinear, skew-symmetric and satisfies Jacobi}\},
$$
$$\mathcal{N}_n = \{\mu \in \mathfrak{L}_n: \mu \mbox{ is nilpotent}\},$$ and $\ad_{\mu}:\RR^n \rightarrow \RR^n$ the adjoint representation of $\mu \in \mathfrak{L}_n$ (i.e. $\ad_{\mu}(x)(y)=\mu(x,y)$).

Then, $\Gl_n(\RR)$ acts on $\mathfrak{L}_n$ by
\begin{equation}\label{ac}
h.\mu(X,Y)=h\mu(h^{-1}X,h^{-1}Y), \quad X,Y \in \RR^n, \quad h \in \Gl_n(\RR), \quad \mu \in \mathfrak{L}_n.
\end{equation}

Each $\mu \in \mathfrak{L}_n$ defines a Lie group endowed with a left-invariant Riemannian metric,
$$
\mu \in \mathfrak{L}_n \rightsquigarrow (G_{\mu},\ip),
$$ where $G_{\mu}$ is the simply connected Lie group with Lie algebra $(\RR^n,\mu)$ endowed with the left-invariant Riemannian metric determined by the inner product $\ip.$ Often, we will denote this metric by $g_{\mu}.$ Note that $g_{\mu}$ may be viewed as a metric on $\RR^{n},$ in fact, $G_{\mu}$ is diffeomorphic to $\RR^n.$

Geometrically, each $h \in \Gl_n(\RR)$ determines a Riemannian isometry
\begin{equation}\label{isom}
(G_{h.\mu},\ip) \rightarrow (G_{\mu}, \la h \cdot,h \cdot \ra),
\end{equation} by exponentiating the Lie algebra isomorphism $h^{-1}: (\RR^n, h.\mu) \rightarrow (\RR^n, \mu)$. Thus the orbit $\Gl_n(\RR).\mu$ parameterizes the set of all left-invariant metrics on $G_{\mu}.$
\begin{definition}\label{solal}
Let $(G,g)$ be a Lie group with a left-invariant Riemannian metric; $g$ is called an algebraic soliton if
\begin{equation}\label{solalg}
\Ricci(g)= cI+ D, \quad \mbox{ for some} \quad c \in \RR,\quad D \in \Der(\ggo),
\end{equation} where $\Ricci(g)$ is the Ricci operator of $g$ and $\ggo$ is the Lie algebra of $G.$
\end{definition} Any homogeneous simply connected algebraic soliton is a Ricci soliton (see \cite[Proposition 3.3]{homRS}).
\subsection{Ricci flow on Lie groups and the bracket flow}\label{frh}
Let $(G,g)$ be a simply connected Lie group endowed with a left-invariant Riemannian metric. Then, if we fix $\ip$ an inner product on the Lie algebra of $G,$ $(G,g)$ is isometric to $(G_{\mu},g_{\mu}),$ for some $\mu \in \mathcal{L}_n.$ In this case, the equation of the Ricci flow (\ref{fr}) is equivalent to the following ordinary differential equation (see \cite[Section 3]{riccifl}):
\begin{equation}\label{frg}
\tfrac{d}{dt}\ip_t=-2\ric(\ip_t),\quad \la\cdot,\cdot\ra_0=\ip,
\end{equation} where $\ric(\ip_t):=\ric(g(t))(e)$ and $e$ is the identity of $G_{\mu}.$ In Subsection \ref{vlb}, we have observed that $\Gl_n(\RR).\mu$ parameterizes the set of all left invariant Riemannian metrics on $G_{\mu},$ then it is very natural to ask: How is the behavior of the Ricci flow in $\mathfrak{L}_n$?
\begin{definition}\label{bflg}
Given $\mu \in \mathcal{L}_n,$ the  bracket flow starting at $\mu$ is the following ordinary differential equation:
\begin{equation}\label{bf}
\tfrac{d}{dt}\mu(t) = \delta_{\mu(t)}\left(\ricci_{\mu(t)}\right)
, \quad \mu(0)=\mu,
\end{equation} where $\delta_{\mu}(A)=\mu(A\cdot,\cdot)+\mu(\cdot,A\cdot)-A\mu(\cdot,\cdot)$, $A \in \G$, $\mu \in \mathcal{L}_n$.
\end{definition}
Let us consider $g(t)$ the Ricci flow solution flow starting at $g_{\mu}$, and $\mu(t)$ the solution of the bracket flow starting at $\mu.$ By \cite{riccifl}, we know that $g(t)$ and $\mu(t)$ are related in the following way.
\begin{theorem}\label{rel}\cite[Theorem 3.3]{riccifl}
There exists time-dependent diffeomorphisms
$$
\varphi(t):G \rightarrow G_{\mu(t)}{\hspace{.2cm}} \mbox{such that}{\hspace{.2cm}} g(t)=\varphi(t)^{*}g_{\mu(t)}, {\hspace{.2cm}} \forall t \in (a,b).
$$
Moreover, if we identify $G=G_{\mu}$, then $\varphi(t):G_{\mu}\rightarrow G_{\mu(t)}$ can be chosen as the equivariant diffeomorphism determined by the Lie group isomorphism between $G_{\mu}$ and $G_{\mu(t)}$ with derivative  $h:\RR^n \rightarrow \RR^n$, where $h(t):=$d$\varphi(t)\mid_{e}:\RR^n \rightarrow \RR^n$ is the solution to any of the following systems of ordinary differential equations:
\begin{enumerate}
  \item $\frac{d}{dt}h=-h\ric(\ip_t)$, $h(0)=I.$
  \item $\frac{d}{dt}h=-\ricci_{\mu(t)}h$, $h(0)=I.$
\end{enumerate}
The following conditions hold:
\begin{itemize}
\item [(3)] $\ip_t=\la h \cdot,h \cdot\ra.$
  \item [(4)] $\mu(t)=h\mu_0(h^{-1} \cdot,h^{-1} \cdot)$.
\end{itemize}
\end{theorem}
\begin{remark}
In this paper, Theorem \ref{rel} has only been stated in the case of Lie groups, however, in \cite{riccifl} it is stated and proved in the general homogeneous case.
\end{remark}
So, the Ricci flow $g(t)$ can be obtained from the bracket flow $\mu(t)$ by solving (2) and applying part (3). In the same way, we can obtain $\mu(t)$ solving (1) and replacing in (4). In particular, both flows are defined in the same interval of time. For more information, see \cite{riccifl}.

We now recall some results proved by J. Lauret in \cite{riccinil} about the Ricci flow for simply connected nilmanifolds.
\begin{theorem}\cite{riccinil}\label{nilva}
Let $\mu(t)$ be the solution bracket flow starting at $\mu \in \mathcal{N}_n,$ and $g(t)$ the Ricci flow starting at $g_{\mu}.$ Then
\begin{itemize}
\item[(i)] $\mu(t)$ is defined for all $t \in [0,\infty).$
\item[(ii)] $g(t)$ is a Type-III solution for a constant $C_n$ that only depends on the dimension $n.$
\item[(iii)] $\mu(t) \rightarrow 0,$ as $t \rightarrow \infty.$ Moreover, $g_{\mu(t)}$ converges in $\mathcal{C}^{\infty}$ to the flat metric $g_0.$
\item[(iv)] $g_{\frac{\mu(t)}{\|\mu(t)\|}}$ converges in $\mathcal{C}^{\infty}$ to an algebraic soliton $g_{\lambda}$ uniformly on compact sets in $\RR^n,$ as $t \rightarrow \infty.$
\end{itemize}
\end{theorem}
\section{The bracket flow in a class of solvmanifolds}\label{nue}
In this section, we study the bracket flow for a metric solvable Lie algebra with an abelian ideal of codimension one.

We consider $(\RR^{n+1},\ip),$ with $\ip$ the canonical inner
product on $\RR^{n+1}.$ If the dimension of the Lie algebra is
$n+1,$ then up to isomorphism, we can assume that the Lie bracket
has the following form with respect to the canonical basis
$\{e_0,e_1,\ldots,e_n\}:$
$$
\mu_{A}(e_0,e_i)=A e_i, {\hspace{.2cm}} i=1, \ldots, n, \quad
\mu_{A}(e_i,e_j)=0,{\hspace{.2cm}} \forall i,j \geq 1, \quad A \in
\glg_n(\mathbb{R}),
$$ where we think of an $A\in \g$ as an operator acting on $\RR^n,$ the subspace generated by $\{e_1, e_2,\ldots,e_n\}$ (i.e. the codimension-one abelian ideal). From now on, we denote these algebras by
$(\RR^{n+1},\mu_{A}),$ or simply, $\mu_{A}.$

\begin{lemma}
If $A_0 \in \g,$ then the bracket flow starting at $\mu_{A_0}$ is
given by $\mu_{A(t)},$ $t \in (T_-, T_+),$ where $A(t)$ satisfies
\begin{equation}\label{a'}
{\tfrac{d}{dt}} A=-\tr(S(A)^{2})A+\tfrac{1}{2}[A,[A,A^{t}]]-\tfrac{1}{2}\tr(A)[A,A^{t}], \quad A(0)=A_0.
\end{equation}
\end{lemma}
\begin{proof}
By using the formula for the Ricci operator of a solvmanifold (see
for instance \cite[Section 4]{riccisol}), we obtain that the Ricci
operator of $(G_{\mu_{A}},g_{\mu_{A}})$ with respect to the basis
$\{e_0,e_1,\ldots,e_n\}$ is represented by the matrix
\begin{equation}\label{ricmu_A}
\ricci_{\mu_{A}}= \left(
  \begin{array}{cc}
    -\tr(S(A)^{2}) & 0 \\
    0 & \tfrac{1}{2}[A,A^{t}]-\tr(A)S(A) \\
  \end{array}
\right),
\end{equation} where $S(A)=\frac{1}{2}(A +A^{t})$ is the symmetric part of the
matrix $A$ and $\tr(A)$ is the trace. Then,
$$
\begin{array}{rcl}
\delta_{\mu_A}(\ricci_{\mu_A})(e_0,e_i)&=&\mu_A(\ricci_{\mu_A}e_0,e_i)+\mu_A (e_0,\ricci_{\mu_A}e_i)-\ricci_{\mu_A}\mu_A(e_0,e_i)\\
&=& -\tr(S(A)^{2})A e_i  + A \ricci_{\mu_A}|_{\RR^n}e_i - \ricci_{\mu_A}|_{\RR^n} A e_i\\
&=& -\tr(S(A)^{2})A e_i + [A,\ricci_{\mu_A}|_{\RR^n}] e_i\\
&=&\left(-\tr(S(A)^{2})A+\tfrac{1}{2}[A,[A,A^{t}]]-\tfrac{1}{2}\tr(A)[A,A^{t}]\right)e_i,
\end{array}
$$ and, on the other hand, we have that
$\delta_{\mu_A}(\ricci_{\mu_A})(e_i,e_j)=0,$ for all $i,j\geq 1,$ as
$\mu_A|_{\RR^n \times \RR^n}=0.$ So,
$$
\delta_{\mu_A}(\ricci_{\mu_A})=\mu_{B}, \quad \mbox{where} \quad B=-\tr(S(A)^{2})A+\tfrac{1}{2}[A,[A,A^{t}]]-\tfrac{1}{2}\tr(A)[A,A^{t}].
$$ Then, this family of Lie algebras is invariant under the bracket flow,
which is equivalent to (\ref{a'}). In addition, the maximal interval
of time where $\mu_{A(t)}$ exists is of the form $(T_-, T_+)$ for
some $-\infty \leq T_- < 0 < T_+ \leq \infty,$ since \eqref{a'} is
an ODE.
\end{proof}
So, given a matrix $A_0,$ we have that the bracket flow starting at
$\mu_{A_0}$ is equivalent to an evolution equation for a curve of
matrices with initial condition $A_0.$ In what follows, we will
often think of the bracket flow as this evolution.
\begin{remark}\label{flat}
Note that the only fixed points of the system (\ref{a'}) are the skew-symmetric
matrices, which are precisely the flat solvmanifolds of the form
$\mu_A$, since by (\ref{ricmu_A}) they are precisely the Ricci-flat ones (see \cite{Alk} and \cite{Mln}).
\end{remark}
\begin{proposition}\label{solmuA}
For any $A_0 \in \g,$ the following conditions are equivalent:
\begin{itemize}
  \item[(i)] $\mu_{A_0}$ is an algebraic soliton.
  \item[(ii)] $A_0$ is either a normal matrix or $A_0$ is a nilpotent matrix such that $[A_0,[A_0,{A_0}^t]]= c A_0,$ for some $c \in \RR.$
\end{itemize} Moreover, the evolution of the bracket flow is respectively given by $$A(t)= \left(2 \tr(S(A_0)^2)t+1\right)^{-1/2}A_0 \quad \mbox{or} \quad A(t)=\left((-\|A_0\|^2+c)t+1\right)^{-1/2}A_0.$$
\end{proposition}
\begin{proof}
Assuming part (i), we have two cases:
\begin{itemize}
  \item If the nilradical of $\mu_{A_0}$ has dimension $n,$ then $A_0$ is a normal matrix (see \cite[Theorem 4.8]{riccisol}).
  \item If the nilradical of $\mu_{A_0}$ has dimension $n+1,$ then $\mu_{A_0}$ is nilpotent and so $A_0$ is a nilpotent matrix. In addition, from (\ref{ricmu_A}), we have that
$$
\left(
  \begin{array}{cc}
    -\frac{1}{2}\|A_0\|^2 & 0 \\
    0 & \tfrac{1}{2}[A_0,{A_0}^{t}] \\
  \end{array}
\right) = \ricci_{\mu_{A_0}}= cI+D,
$$
\end{itemize} and it follows that $D(e_0)=\lambda e_0.$ Also, we know that $[\ad_{\mu_{A_0}}(e_0),D]=-\ad_{\mu_{A_0}}(D(e_0)),$ so
{\small $$\left(
    \begin{array}{cc}
      0 & 0 \\
      0 & \frac{1}{2}[A_0,[A_0,{A_0}^t]] \\
    \end{array}
  \right)
 = [\ad_{\mu_{A_0}}(e_0),\ricci_{\mu_{A_0}}] = [\ad_{\mu_{A_0}}(e_0),D] = -\lambda \ad_{\mu_{A_0}}(e_0) = -\lambda \left(
                                     \begin{array}{cc}
                                       0 & 0 \\
                                       0 & A_0 \\
                                     \end{array}
                                   \right).
$$}
Conversely, if $A_0$ is a normal matrix, then $\mu_{A_0}$ is an algebraic soliton (see \cite[Theorem 4.8]{riccisol}) and if $A_0$ is a nilpotent matrix which satisfies $[A_0,[A_0,{A_0}^t]]= c A_0,$ then
$$
 \ricci_{\mu_{A_0}} = \left(
  \begin{array}{cc}
    -\frac{1}{2}\|A_0\|^2 & 0 \\
    0 & \tfrac{1}{2}[A_0,{A_0}^{t}] \\
  \end{array}
\right) = \tfrac{c-\|A_0\|^2}{2}I+\left(
                                         \begin{array}{cc}
                                           -\tfrac{c}{2} & 0 \\
                                           0 & -\tfrac{c}{2}I+\tfrac{1}{2}\|A_0\|^2I+\tfrac{1}{2}[A_0,{A_0}^t] \\
                                         \end{array}
                                       \right),
$$ and it is easy to see that $\left(
                                         \begin{array}{cc}
                                           -\tfrac{c}{2} & 0 \\
                                           0 & -\tfrac{c}{2}I+\tfrac{1}{2}\|A_0\|^2I+\tfrac{1}{2}[A_0,{A_0}^t] \\
                                         \end{array}
                                       \right)$ is a derivation of
$\mu_{A_0},$ and so (i) is proved.

Finally, if $\mu_{A_0}$ is an algebraic soliton, then the family $A(t)=a(t)A_0$ is invariant under the flow. Therefore, we have that
\begin{itemize}
\item If $A_0$ is a normal matrix, then the bracket flow is equivalent to the following differential equation for $a=a(t)$
$$a'=-\tr(S(A_0)^2)a^3, \quad a(0)=1,$$ and so the solution is $A(t)= (2 \tr(S(A_0)^2)t+1)^{-1/2}A_0.$
\item If $A_0$ is a nilpotent matrix, then the bracket flow is equivalent to
$$a'=\frac{-\|A_0\|^2+c}{2}a^3, \quad a(0)=1,$$ and so the solution is $A(t)=((-\|A_0\|^2+c)t+1)^{-1/2}A_0.$
\end{itemize}
\end{proof}
The first natural question that arises is related with the maximal time interval of the solution $A(t).$ An important point to observe here
is that $-\infty<T_- $ since $(G_{\mu_{A(t)}}, g_{\mu_{A(t)}})$
always has non-positive scalar curvature (see (\ref{ricmu_A})).
\begin{proposition}\label{a}
$A(t)$ is always defined for all $t \in [0,\infty)$ (i.e. $T_+=\infty$).
\end{proposition}
\begin{proof}
By using (\ref{a'}), we get
$$
{\tfrac{d}{dt}} \|A\| ^{2}= 2\langle A,-\tr(S(A)^{2})A \rangle+
2\langle A,\tfrac{1}{2}[A,[A,A^{t}]] \rangle- 2\langle
A,\tfrac{1}{2}\tr(A)[A,A^{t}])\rangle.
$$ But since $\langle A,[A,[A,A^{t}]]\rangle=-\|[A,A^{t}]\|^{2}$ and
$\langle A,[A,A^{t}]\rangle=0,$ it follows that
\begin{equation}\label{normA}
{\tfrac{d}{dt}} \|A\| ^{2}=-2\tr(S(A)^{2})\|A\| ^{2}-\|[A,A^{t}]\|^{2} \leq 0.
\end{equation}
Therefore, $\|A\| ^{2}$ decreases and so $A(t)$ is defined for all
$t \in [0,\infty),$ as the solution remains in a
compact subset.
\end{proof}
\begin{remark}
By Theorem \ref{rel} and the previous proposition, we obtain that
the Ricci flow starting at any of these solvmanifolds
$(G_{\mu_{A_0}},\ip)$ is defined for $t \in [0,\infty),$ often
called an immortal solution.
\end{remark}
In what follows, we introduce a positive, non-increasing function
along the normalized bracket flow $\tfrac{A(t)}{\|A(t)\|},$ which is
strictly decreasing unless $\mu_{\tfrac{A_0}{\|A_0\|}} $ is an
algebraic soliton. The advantage of having this function lies in the
fact that it will allow us to prove that for any sequence $t_k \rightarrow \infty$ there exists a subsequence in which the normalized bracket flow
always converges to an algebraic soliton.
\begin{lemma}\label{func}
Let $\mu_{A(t)}$ be the bracket flow starting at $\mu_{A_0}$ and set
$B(t)=\tfrac{A(t)}{\|A(t)\|}.$ Then $F(B)=\|[B,B^t]\|^2$ is a
positive, non-increasing function along the flow. Moreover, $\tfrac{d}{dt}
|_{t=t_0}F (B) = 0,$ for some $t_0,$ if and only if $\mu_{B(0)}$
is an algebraic soliton.
\end{lemma}
\begin{proof}
We consider $F: \g \rightarrow \RR,$
$$
F(C)=\| [C,C^t]\|^2.
$$ Then
$$
\tfrac{d}{dt}F \left(\tfrac{A}{\|A\|}\right) = \tfrac{d}{dt}
\tfrac{\|[A,A^t]\|^2}{\|A\|^4}= \tfrac{1}{\|A\|^8} (\|A\|^4
\tfrac{d}{dt} \|[A,A^t]\|^2 - \|[A,A^t]\|^2 \tfrac{d}{dt}\|A\|^4).
$$ By using the bilinearity of the inner product and the lie bracket we obtain that
$$
\tfrac{d}{dt} \|[A,A^t]\|^2=-4\tr(S(A)^2) \|[A,A^t]\|^2 - 2 \|[A,[A,A^t]]\|^2,
$$ and from (\ref{normA})
$$
\tfrac{d}{dt}\|A\|^4 = 2 \|A\|^2 \tfrac{d}{dt}\|A\|^2 = -4\tr(S(A)^{2})\|A\|^{4} - 2 \|A\|^2 \|[A,A^{t}]\|^{2}.
$$ Then, if we consider $B = \tfrac{A}{\|A\|},$ it follows that
\begin{equation}\label{F'}
\tfrac{d}{dt}F \left(\tfrac{A}{\|A\|}\right) = 2 \|A\|^2
(\|[B,B^t]\|^4- \|B\|^2 \|[B,[B,B^{t}]]\|^{2}) \leq 0,
\end{equation} by using the Cauchy-Schwarz inequality. Moreover, if there exists $t_0 \in \RR$ such that $\tfrac{d}{dt} |_{t=t_0}F \left(\tfrac{A}{\|A\|}\right)=0,$ then the Cauchy-Schwarz equality holds and there exists $c \in \RR$ such that
\begin{equation}\label{cse}
[B(t_0),[B(t_0),B(t_0)^t]] = c B(t_0).
\end{equation}
We have two cases:
\begin{itemize}
\item If $c=0,$ then $[B(t_0),[B(t_0),B(t_0)^t]]=0,$ and so $\tr([B(t_0),[B(t_0),B(t_0)^t]]B(t_0)^t)=0$, this implies that $\|[B(t_0),B(t_0)^t]\|^2=0$, i.e. $B(t_0)$ is normal and $\mu_{B(t_0)}$ is an algebraic soliton (see Proposition \ref{solmuA}). On the other hand, by using (\ref{a'}) and
(\ref{normA}), it is easy to see that
\begin{equation}\label{a'norm}
\begin{array}{rcl}
\tfrac{d}{dt} B &=& \tfrac{d}{dt} \tfrac{A}{\|A\|} = \tfrac{1}{2
\|A\|} \left( [A,[A,A^t]] - \tr(A) [A,A^t] +
\tfrac{\|[A,A^t]\|^2}{\|A\|^2}A \right)\\ &=& \tfrac{\|A\|^2}{2}
\left( [B,[B,B^t]] - \tr(B) [B,B^t] + \|[B,B^t]\|^2B \right),
\end{array}
\end{equation} so, $B(t) = B(t_0),$ for all $t,$ since $B(t_0)$ is a fixed point of (\ref{a'norm}). It follows that $\mu_{B(t)}=\mu_{B(t_0)}$ for all $t.$
\item If $c \neq 0,$ then by using (\ref{cse}), we obtain that $\tr(B(t_0))=0$ and $\tr(B(t_0)^k)=0,$ since
$$
\begin{array}{rcl}
c\tr(B(t_0)^{k+1})&=&\tr([B(t_0),[B(t_0),B(t_0)^t]]B(t_0)^k)\\
&=&\tr([B(t_0)^k,B(t_0)][B(t_0),B(t_0)^t])\\
&=&0.
\end{array}
$$ Therefore, $B(t_0)$ is a nilpotent matrix that satisfies (\ref{cse}), so by Proposition \ref{solmuA} we have that $\mu_{B(t_0)}$ is an algebraic soliton. In addition, $B(t_0)$ is a fixed point of (\ref{a'norm}), so, $\mu_{B(t)}=\mu_{B(t_0)}$ for all $t.$
\end{itemize}
Conversely, if $\mu_{B(0)}$ is an algebraic soliton, then by using
(\ref{F'}), we have that $\tfrac{d}{dt}F
\left(\tfrac{A}{\|A\|}\right)=0.$
\end{proof}
\begin{corollary}\label{notafunc}
Let $\mu_{A(t)}$ be the bracket flow starting at $\mu_{A_0}$ and set
$B(t)=\tfrac{A(t)}{\|A(t)\|}.$ Then for any sequence $t_k \rightarrow \infty$ there exists a subsequence of
$(G_{\mu_{B(t_k)}}, g_{\mu_{B(t_k)}})$ converging in the pointed
topology to an algebraic soliton
$(G_{\mu_{B_{\infty}}},g_{\mu_{B_{\infty}}}).$
\end{corollary}
\begin{proof}
Every sequence $B(t_k)$ has a convergent subsequence, i.e. after
passing to a subsequence, $B(t_k)$ converges to a matrix
$B_{\infty}.$ Then $\mu_{B_{\infty}}$ is an algebraic soliton by
Lemma \ref{func}, as $B_{\infty}$ is a fixed point of the flow.
\end{proof}
From now on, our purpose is to study the ODE (\ref{a'}). We emphasize that our aim is not to solve the ODE, we are interested in understanding the qualitative behavior of the solution along the time, which is not trivial to predict even when $n$ is very small. In the next lemma we study how it evolves.
\begin{lemma}\label{ffc}
The bracket flow $\mu_{A(t)}$ starting at $\mu_{A_0}$ has the form
\begin{equation}\label{fc}
A(t)=a(t)\varphi_t A_0 \varphi^{-1}_t,
\end{equation}
where $a(t)$ is a positive, non-increasing, real valued function, and $\varphi_t \in \G$ for each $t.$
\end{lemma}
\begin{proof}
If $h(t)= \left(
             \begin{array}{cc}
               b(t) & 0 \\
               0 & \varphi_t \\
             \end{array}
           \right) \in \Gl_{n+1}(\RR)
$, with $b(t)$ a real function and $\varphi_t \in \G$, then
$$
-\ricci_{\mu_{A(t)}}h(t)=-\left(
                     \begin{array}{cc}
                       -\tr(S(A(t))^{2})b(t) & 0 \\
                       0 & (\frac{1}{2}[A(t),A(t)^{t}]-\tr(A(t))S(A(t)))\varphi_t \\
                     \end{array}
                   \right).
$$ The map $h$ given in part (2) of Theorem \ref{rel} has therefore the form $$h(t)= \left(
             \begin{array}{cc}
               b(t) & 0 \\
               0 & \varphi_t \\
             \end{array}
           \right),
$$ and it follows from (4) of the same theorem that $$\mu_{A(t)}=h(t). \mu_{A_0}=\mu_{\frac{1}{b(t)}\varphi_t A_0 \varphi_t^{-1}}.$$ In addition, 
$$
\left\{
  \begin{array}{ll}
    b'(t)= \tr(S(A(t))^{2})b(t),\\
    b(0)=1,
  \end{array}
\right.
$$ so, we have that $b$ is a positive, non-decreasing function. It follows that if $a(t)=\tfrac{1}{b(t)},$ then $a(t)$ is a positive, non-increasing function.
\end{proof}
In what follows, $\mu_{A(t)},$ $A(t) \in \g,$ will be the bracket
flow solution starting at $\mu_{A_0}$ and we will denote it simply by $A(t).$
\begin{proposition}\label{spe}
Assume that $A(t_k) \rightarrow A_{\infty},$ for some sequence $t_k \rightarrow \infty.$ Then $\Spec(A_{\infty})$ $=$
$a_{\infty}\Spec(A_0),$ for some $a_{\infty} \in \RR.$ Here $\Spec(B)$ denotes the unordered set of $n$ complex eigenvalues of the matrix $B \in \g.$
\end{proposition}
\begin{proof}
We know that $A(t)=a(t)\varphi_t A_0 \varphi_t^{-1}$ by Lemma \ref{ffc}, therefore
$$\Spec(A(t_k)) = a(t_k)\Spec(\varphi_{t_k} A_0 \varphi_{t_k}^{-1})=
a(t_k)\Spec(A_0), \quad \forall t_k \in (T_-,\infty).$$ Then, as $A(t_k) \rightarrow A_{\infty},$ we have that
$$\Spec(A_{\infty}) = a_{\infty}\Spec(A_0),$$ where $a_{\infty}= \lim_{k \rightarrow \infty} a(t_k)$ (recall that from Lemma \ref{ffc}, $a(t)$ is a positive, non-increasing function).
\end{proof}
\begin{proposition}\label{sa}
$\tr(S(A(t))^{2})$ is strictly decreasing if $A_0$ is not skew-symmetric.
Moreover, $\tr(S(A(t))^{2})\rightarrow 0,$ as $t \rightarrow
\infty$.
\end{proposition}
\begin{proof}
Recall that $S(A)= \frac{1}{2}(A+A^{t}),$ and so
\begin{equation}\label{trsa}
\tr(S(A)^{2})=\tfrac{1}{2}\|A\| ^{2}+\tfrac{1}{2}\tr(A^{2}).
\end{equation}
Then, as in Proposition \ref{a} we have already studied ${\tfrac{d}{dt}} \|A\|
^{2},$ we will only analyze ${\tfrac{d}{dt}} \tr(A^{2}).$ By using (\ref{a'}), we obtain
\begin{equation}\label{tra2}
{\tfrac{d}{dt}} \tr(A^{2})={\tfrac{d}{dt}}\langle A,A^{t}\rangle=
\langle {\tfrac{d}{dt}} A,A^{t}\rangle + \langle A,
{\tfrac{d}{dt}} A^{t}\rangle = 2 \langle {\tfrac{d}{dt}} A,A^{t}\rangle =-\tr(S(A)^{2})\tr(A^{2}).
\end{equation} Therefore, it follows from (\ref{normA}) and (\ref{tra2}) that
$$
{\tfrac{d}{dt}} \tr(S(A)^{2})=-2\tr(S(A)^{2})^{2}-\tfrac{1}{2}\|[A,A^{t}]\|^{2} \leq 0,
$$ and if there exists $t_0$ such that $\tfrac{d}{dt} |_{t = t_0} \tr(S(A)^{2}) =
0,$ then $A(t_0)$ is a skew-symmetric and so $A(t) = A(t_0),$ for
all $t.$ Conversely if $A_0$ is skew-symmetric, we have that
$\tfrac{d}{dt} \tr(S(A)^{2}) = 0.$ So, $\tr(S(A)^{2})$ is
strictly decreasing if $A_0$ is not skew-symmetric.

In addition,
$$
{\tfrac{d}{dt}} \tr(S(A)^{2}) \leq -2\tr(S(A)^{2})^{2},
$$
And then $\tr(S(A)^{2})$ is dominated by
$$
x(t)= {\tfrac{1}{2t+({\tr(S(A(0))^{2})})^{-1}}},
$$ which is a solution of ${\tfrac{d}{dt}} x=-2x^{2}.$ Therefore $\tr(S(A(t))^{2}) \rightarrow 0,$ as $t
\rightarrow \infty.$
\end{proof}
Recall that if $G_{\mu_A}$ is the simply connected solvable Lie group with Lie algebra $(\RR^{n+1},\mu_A),$ then $g_{\mu_A}$ denotes the left-invariant Riemannian metric on $G_{\mu_A}$
such that $g_{\mu_A}(e)=\langle \cdot, \cdot \rangle$, where $e$
is the identity of the group $G_{\mu_A}$ and $\langle \cdot, \cdot
\rangle$ is the canonical inner product on $\RR^{n+1}$.
\begin{corollary}\label{subconv}
If $A(t) \rightarrow A_{\infty},$ as $t \rightarrow \infty,$ then $A_{\infty}$ is a skew-symmetric matrix and for any sequence $t_k \rightarrow \infty$ there exists a subsequence of $(G_{\mu_{A(t_k)}}, g_{\mu_{A(t_k)}})$ which converges in the pointed topology to a
manifold locally isometric to $(G_{\mu_{A_{\infty}}},
g_{\mu_{A_{\infty}}})$, which is flat.
\end{corollary}
\begin{proof}
By Proposition \ref{sa}, we know that $S(A(t)) \rightarrow 0,$
as $t\rightarrow \infty$, therefore $A_{\infty}$ is
skew-symmetric and then $(G_{\mu_{A_{\infty}}},
g_{\mu_{A_{\infty}}})$ is flat (see Remark \ref{flat}).

Finally, since $\mu_{A(t)} \rightarrow \mu_{A_{\infty}},$ by
\cite[Corollary 6.20]{conv}, for any sequence $t_k \rightarrow \infty$ there exists a subsequence of $(G_{\mu_{A(t_k)}},
g_{\mu_{A(t_k)}})$ which converges in the pointed topology to a
manifold locally isometric to $(G_{\mu_{A_{\infty}}},
g_{\mu_{A_{\infty}}})$, which is flat, as shown above.
\end{proof}
In the following proposition, we prove that under an additional hypothesis, the convergence is actually smooth.
\begin{proposition}\label{conva}
If $\Spec(A_0) \nsubseteq i \RR$ and $A(t) \rightarrow
A_{\infty},$ as $t \rightarrow \infty,$ then $g_{\mu_{A(t)}} \rightarrow
g_{\mu_{A_{\infty}}}$ smoothly on $\RR^{n+1}.$
\end{proposition}
\begin{proof}
For each $\mu_A,$ we define $\psi: \RR \oplus \RR^{n} \rightarrow
G_{\mu_A}$ by
\begin{equation}\label{exp}
\psi(r,x)=\exp_{\mu_A}(re_0)\exp_{\mu_A}(x), \quad r\in \RR, \quad
x\in \RR ^{n},
\end{equation} where $\exp_{\mu_A}: (\RR \oplus \RR^{n}, \mu_A)
\rightarrow G_{\mu_A}$ is the Lie exponential of $G_{\mu_A}.$

Let $\varphi: (\RR^{n+1},\mu_A) \rightarrow
(\glg_{n+1}(\RR),[\cdot, \cdot])$ be the linear transformation
such that $\varphi(e_0)=X_0$ and $\varphi(e_i)=X_i$,
$i=1,\ldots,n,$ where
$X_0=\left(%
\begin{array}{cc}
A & 0 \\
0 & 0 \\
\end{array}%
\right)$ and $X_i=\left(%
\begin{array}{cc}
0 & \tilde{e}_i^{t} \\
0 & 0 \\
\end{array}%
\right)$, $\tilde{e}_i =(0,\ldots,1,\dots, 0) \in \RR ^{n}$. Then
$\varphi$ is an isomorphism of Lie algebras.

Therefore, under the isomorphism $\varphi$, we have that
$$
\psi(r,x)=\exp(rX_0)\exp(x), \quad r\in \RR, \quad x\in \RR ^{n},
$$ where $\exp$ is the exponential function of matrices.

Then
$$
\psi(r,x)=\exp(rX_0)\exp(x)= \exp(rX_0)\exp(x_1X_1+\ldots+x_nX_n),
$$ but $\exp(rX_0)=\left(%
\begin{array}{cc}
  \exp(rA) & 0 \\
  0 & 1 \\
\end{array}%
\right)$ and $\exp(x)=\left(%
\begin{array}{cc}
  I & x \\
  0 & 1 \\
\end{array}%
\right)$, therefore
$$
\psi(r,x)=\left(%
\begin{array}{cc}
  \exp(rA) & \exp(rA)x \\
  0 & 1 \\
\end{array}%
\right).
$$
It is easy to see that if $\Spec(A)\nsubseteq i\RR$ or $A = 0,$ then $\psi$ is a diffeomorphism. So, as $\Spec(A_0) \nsubseteq i \RR,$
we have that $\Spec(A(t)) \nsubseteq i \RR$ and $\Spec(A_{\infty})
\nsubseteq i \RR$ or $A_{\infty}=0$ by Proposition \ref{spe}, and therefore we have
that $g_{\mu_{A(t)}} \rightarrow g_{\mu_{A_{\infty}}}$ smoothly on
$\RR^n$ (see \cite[Remark 6.11]{conv}).
\end{proof}
\begin{remark}
In particular, if $\mu_{A_0}$ is completely solvable ($\Spec(\ad_{\mu}x) \subseteq
\RR,$ for all $x$), then the convergence is smooth. This also follows by using Proposition
\ref{spe} and \cite[Corollary 6.20]{conv}, since $\mu_{A(t)}$ is
completely solvable for all $t.$
\end{remark}
Recall that if the norm of the Riemann tensor decays at least as fast as $\tfrac{C}{t},$ where $C$ is a constant, then the solution of the Ricci flow is a Type-III solution (see Definition \ref{type3}).
\begin{proposition}\label{type3}
For every $\mu_{A_0}$ with $\tr({A_0}^2) \geq 0$, the Ricci flow $g(t)$ with
$g(0)=g_{\mu_{A_0}}$ is a Type-III solution, for some constant $C_{n+1}$
that only depends on the dimension $n+1.$
\end{proposition}
\begin{proof}
In Proposition \ref{a}, we proved that $\mu_{A(t)}$ is defined for $t \in [0,\infty).$ We observe that, by using (\ref{tra2}), if
$\tr({A_0}^2) \geq 0$ then $\tr(A(t)^2) \geq 0$ for all $t.$
Further, in Proposition \ref{sa}, we prove that $\tr(S(A)^2)\leq
{\tfrac{1}{2t+({\tr(S(A(0))^{2})})^{-1}}},$ therefore, by using
(\ref{trsa}), we have that
$$
\|\Riem(\mu_A)\| = \|\mu_A\|^2 \|\Riem(\tfrac{\mu_A}{\|\mu_A\|})\| = 2 \|A\|^2 \|\Riem(\tfrac{\mu_A}{\|\mu_A\|})\| \leq {\tfrac{4 C}{2t+({\tr(S(A(0))^{2})})^{-1}}} \leq \tfrac{2C}{t},
$$
where $C$ is  the maximum of the continuous function $\mu
\rightarrow \|\Riem(\mu)\|$ restricted to the unit sphere of
$\mathfrak{L}_{n+1}.$
\end{proof}
The question that naturally arises is whether the flow converges. The following section is devoted to study such question.
\section{Limit points}\label{punlim}
In this section, we analyze the $\omega$-limit of the bracket flow
$\mu_{A(t)}$ (i.e. the set of limit points of sequences under the bracket flow). To do this, we consider two cases: when $\tr(A_0)=0$
( i.e., $\mu_{A_0}$ is unimodular) and when $\tr(A_0) \neq 0.$

Let us first suppose that $\tr(A_0)=0.$

We consider the functional $F(A)=\|[A,A^{t}]\|^{2},$ which is, in fact, the square norm of the moment map of the conjugation action of the real reductive group $\G$ on the vector space $\g,$ and we compute its gradient:
$$
\begin{array}{rcl}
\la grad(F)_A , B \ra &=& \tfrac{d}{dt}|_{t=0} F(A+tB)= \tfrac{d}{dt}|_{t=0} \|[A+tB,A^t+tB^{t}]\|^{2}\\
                      &=& \tfrac{d}{dt}|_{t=0} \la [A+tB,A^t+tB^{t}], [A+tB,A^t+tB^{t}]\ra\\
                      &=& 2 \la [A,A^t], \tfrac{d}{dt}|_{t=0} [A+tB,A^t+tB^{t}]\ra\\
                      &=& 2 \la [A,A^t], [B,A^t]+[A,B^t] \ra\\
                      &=& 4 \la [A,A^t], [B,A^t] \ra = -4 \la [A,[A,A^t]], B \ra.
\end{array}
$$ Thus, $grad(F)_A = -4 [A,[A,A^t]]$ and the negative gradient flow of $F$ is given by
\begin{equation}\label{abarra}
{\tfrac{d}{dt}}\bar{A}(t)= 4[\bar{A}(t),[\bar{A}(t),\bar{A}(t)^{t}]].
\end{equation}
Observe that $\|\bar A\|$ is a decreasing function. Indeed,
$$
{\tfrac{d}{dt}} \|\bar A\|^2= 2 \langle {\bar A}', \bar A \rangle = 8 \la [\bar{A},[\bar{A},\bar{A}^{t}]],\bar A  \ra = -8 \|[\bar A,\bar A^{t}]\|^{2},
$$ as $\langle[\bar A,[\bar A,\bar A^{t}]], \bar A \rangle =
-\|[\bar A,\bar A^{t}]\|^{2}.$ So, $\bar A (t)$ has a limit point $A_{\infty}^{1}$ and then we have that there exists the limit of $\bar A(t),$ as $t \rightarrow \infty$ and it is unique (see \cite[Introduction]{Krd}). In addition, if $\bar A(t)\rightarrow A_{\infty}^{1},$ we have two cases:
\begin{itemize}
\item If $A_{\infty}^{1} \neq 0,$ then $\lim_ {t\rightarrow \infty}\frac{\bar{A}(t)}{\|\bar{A}(t)\|}$ exists and $\lim_ {t\rightarrow \infty}\frac{\bar{A}(t)}{\|\bar{A}(t)\|}=\tfrac{A_{\infty}^{1}}{\|A_{\infty}^{1}\|}.$
\item If $A_{\infty}^{1} = 0,$ then by \cite[Theorem 7.1]{Krd}, $\lim_ {t\rightarrow \infty}\frac{\bar{A}(t)}{\|\bar{A}(t)\|}$ exists.
\end{itemize}
If $A_0$ is nilpotent, then $\mu_{A_0}$ turns out to be nilpotent and so
the bracket flow starting at $\mu_{A_0}$ has been studied in
\cite{riccinil} (see Theorem \ref{nilva}). Therefore, we assume that $A_0$ is not nilpotent.
\begin{lemma}\label{cociente}
Assume that $\tr(A_0)=0$ and $A_0$ is not nilpotent. Let $\mu_{A(t)}$ be the bracket
flow starting at $\mu_{A_0}$ and let $\bar A(t)$ be the negative gradient
flow (\ref{abarra}) starting at $A_0.$ Then the limit of $\tfrac{A(t)}{\|A(t)\|}$ exists and
$$\lim_ {t\rightarrow \infty}\frac{A(t)}{\|A(t)\|} = \lim_ {t\rightarrow \infty}\frac{\bar{A}(t)}{\|\bar{A}(t)\|}.$$
\end{lemma}
\begin{proof}
We prove that, up to scaling and reparameterization of the time,
the bracket flow $A(t)$ starting at $A_0$ is $\bar{A}(t),$
the solution of (\ref{abarra}) starting at $A_0,$ i.e. we want to
show that there exist $c(t)$ and $\tau(t)$ such that
$A(t)=c(t)\bar{A}(\tau(t)).$

Let $c(t)$ and $\tau(t)$ be solutions of the following system of
differential equations with initial conditions:
$$
\begin{array}{ll}
c'(t)=-\tr(S(\bar{A}(\tau(t))^2))c(t)^3, & c(0)=1,\\
\tau'(t)= \tfrac{1}{8} c(t)^{2}, &
\tau(0)=0.
\end{array}
$$ It is easy to see that $c(t)$ and $\tau(t)$ are defined for all
$t,$ and with a simple calculation it is easy to verify that
$c(t)\bar{A}(\tau(t))$ is a solution of the equation (\ref{a'}),
therefore by uniqueness
$$A(t)=c(t)\bar{A}(\tau(t)), \quad \forall t \in [0,\infty).$$
If $\tau(t)\rightarrow \infty$ then
$$\lim_ {t\rightarrow \infty} \frac{A(t)}{\|A(t)\|}=\lim_ {t\rightarrow \infty} \frac{\bar{A}(\tau(t))}{\|\bar{A}(\tau(t))\|}=\lim_ {t\rightarrow \infty} \frac{\bar{A}(t)}{\|\bar{A}(t)\|}.$$ We suppose that $\tau(t)\rightarrow L$, $L<\infty,$ as
$t\rightarrow\infty$, then
$$\lim_ {t\rightarrow \infty} \frac{A(t)}{\|A(t)\|}=\lim_ {t\rightarrow \infty} \frac{\bar{A}(\tau(t))}{\|\bar{A}(\tau(t))\|}=\frac{\bar{A}(L)}{\|\bar{A}(L)\|}.$$ This implies that $\tfrac{\bar{A}(L)}{\|\bar{A}(L)\|}$ is an algebraic soliton, since it is the limit of a normalized bracket flow (see \cite[Proposition 4.1]{homRS}). As, $A_0$ is not nilpotent and $\bar{A}(t)$ is conjugated to $A_0,$ for each $t,$ we have then $\tfrac{\bar{A}(L)}{\|\bar{A}(L)\|}$ is normal (see Proposition \ref{solmuA}), i.e. $\bar{A}(L)$ is
normal. So, $\bar A (t)= \bar A (L),$ for all $t \geq L,$ by (\ref{abarra}).

Therefore,
$$\lim_ {t\rightarrow \infty} \frac{A(t)}{\|A(t)\|}=\lim_ {t\rightarrow \infty} \frac{\bar{A}(\tau(t))}{\|\bar{A}(\tau(t))\|}=\lim_ {t\rightarrow \infty} \frac{\bar{A}(t)}{\|\bar{A}(t)\|},$$ as was to be shown.
\end{proof}
\begin{remark}\label{nochaos}
It follows from Lemma \ref{cociente} and \cite[Section 7]{riccinil} that if $\mu_{A_0}$ is unimodular, i.e. $\tr(A_0) = 0,$ then the $\omega$-limit of $\frac{A(t)}{\|A(t)\|}$ is a single point.
\end{remark}
\begin{lemma}\label{trneq0}
If $\tr(A_0) \neq 0,$ then $A(t) \rightarrow 0,$ as $t\rightarrow
\infty.$
\end{lemma}
\begin{proof}
We know that $A(t)=a(t) \varphi _t A_0 \varphi_t^{-1}$ by Lemma \ref{ffc}, therefore
$$
\tr(A(t))=a(t)\tr(A_0).
$$ If $A(t_k) \rightarrow A_{\infty}$, then
$$
a(t_k)\tr(A_0)=\tr(A(t_k)) \rightarrow \tr(A_{\infty})=0,
$$ so, as $\tr(A_0) \neq 0,$ we have that $a(t_k) \rightarrow 0.$

On the other hand,
$$
\Spec(A(t_k))=a(t_k)\Spec(A_0)\rightarrow \Spec(A_{\infty}),
$$ and so $\Spec(A_{\infty})=0.$ Then $A_{\infty}=0,$ since $A_{\infty}$ is a skew-symmetric matrix.
\end{proof}
By using the two previous lemmas, we can prove the following theorem, which provides information about the
$\omega$-limit of $\mu_{A_0},$ for any $A_0 \in \g.$
\begin{theorem}
The $\omega$-limit of $\mu_{A_0}$ is a single point, for any $A_0 \in \g.$
\end{theorem}
\begin{proof}
By Lemma \ref{trneq0}, we have that if $\tr(A_0) \neq 0,$ then
$A(t)\rightarrow 0,$ as $t \rightarrow \infty.$ If $\tr(A_0)=0$, we know by Remark \ref{nochaos} that
$$
\lim_ {t\rightarrow \infty} \frac{A(t)}{\|A(t)\|}= A_{\infty}^2.
$$
Then, we have that $A(t)\rightarrow A_{\infty},$ as $t\rightarrow
\infty.$ Indeed, the norm of $A(t)$ decreases and therefore $\lim_ {t\rightarrow \infty} \|A(t)\|=\alpha.$ If $\alpha=0,$ then $A(t) \rightarrow 0$ and if $\alpha>0,$ we have that
$$
\lim_ {t\rightarrow \infty} A(t) = \lim_ {t\rightarrow \infty} \alpha \frac{A(t)}{\|A(t)\|} = \alpha A_{\infty}^2,
$$ which completes the proof.
\end{proof}
All results obtained so far can be summarized in the following theorem.
\begin{theorem}\label{summ}
Given $A_0 \in \glg_n(\RR),$ consider the bracket flow
$\mu_{A(t)}$ starting at $\mu_{A_0}$ and $g(t)$ the Ricci flow starting at
$g_{\mu_{A_0}}$. Then,
\begin{itemize}
    \item[(i)] $g(t)$ is defined for $t \in (T_-,\infty),$ where $-\infty < T_-< 0.$
    \item[(ii)] The $\omega$-limit of $\mu_{A_0}$ is a single point.
    \item[(iii)] For any sequence $t_k \rightarrow \infty,$ there exists a subsequence of $(G_{\mu_{A(t_k)}}, g_{\mu_{A(t_k)}})$ which converges in the pointed topology to a manifold locally isometric to $(G_{\mu_{A_{\infty}}}, g_{\mu_{A_{\infty}}})$, which is flat.
    \item[(iv)] If $\Spec(A_0) \nsubseteq i \RR,$ then $g_{\mu_{A(t)}} \rightarrow g_{\mu_{A_{\infty}}}$ smoothly on $\RR^n.$
    \item[(v)] If $\tr({A_0}^2) \geq 0,$ the Ricci flow $g(t)$ with $g(0)=g_{\mu_{A_0}}$ is a Type-III solution, for some constant $C_{n+1}$ that only depends on the dimension $n+1.$
\end{itemize}
\end{theorem}
\begin{example}\label{asecir}
Let $A_0 = \left(
       \begin{array}{cc}
         0 & x_0 \\
         y_0 & 0 \\
       \end{array}
     \right).
$ It is easy to see that the family of matrices of this kind is invariant under the flow (\ref{a'}), which is equivalent to the following ODE system for the variables $x(t),$ $y(t):$
\begin{equation}\label{bfasecir}
\left\{\begin{array}{l}
x'=x(x+y)(-\tfrac{3}{2}x+\tfrac{1}{2}y), \quad x(0)=x_0,\\
y'=y(x+y)(-\tfrac{3}{2}y+\tfrac{1}{2}x), \quad y(0)=y_0.
\end{array}\right.
\end{equation}
\begin{figure}[h]\label{fcxy} 
\includegraphics[width=200pt]{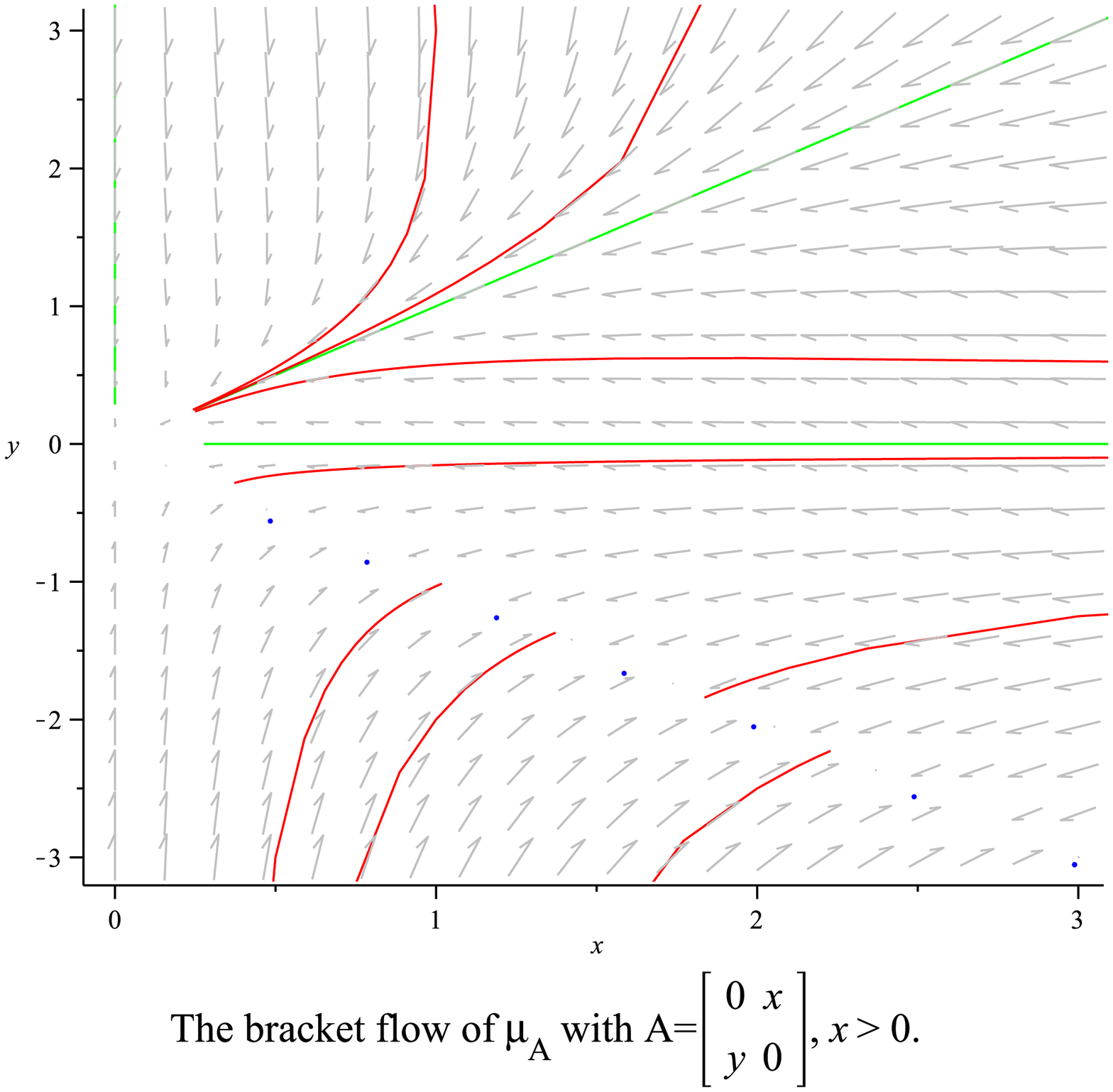}
\caption{}\label{Figure 1}
\end{figure}
The phase plane for this system is displayed in Figure \ref{Figure 1}, as computed in Maple.
It is easy to see that it is enough to assume $0 \leq x_0,$ since if $(x,y)$ is the solution starting at $(x_0,y_0),$ then $(-x,-y)$ is the solution starting at $(-x_0,-y_0).$

Regarding the interval of definition, the solutions remain in a compact subset and so they are defined in $[0,\infty).$

The solutions converge to the points $(x_{\infty},-x_{\infty}),$ which are precisely the fixed points of the system and correspond to skew-symmetric matrices (which in turn correspond to flat metrics). Also, we observe that points of the form $(x_0,x_0),$ $(x_0,0)$ and $(0,y_0)$ correspond to algebraic solitons (they are symmetric or special nilpotent matrices). Despite the fact that the solutions in the upper half-plane converge to $0,$ we can see from the figure that they are approaching the soliton line $y=x$, so considering a suitable normalization we may be able to obtain convergence of those solutions to a non-flat algebraic soliton. This will be the topic of the next section.
\end{example}
\section{Normalizing by the bracket norm}\label{nor}
According to Theorem \ref{summ} (iii), for any sequence $t_k \rightarrow \infty$ there exists a subsequence in which the Ricci flow converge in the pointed topology to a flat manifold. In order to avoid this
type of convergence and get a more interesting limit, we consider
different normalizations of the flow. In this section, we study
the normalized bracket flow by the bracket norm, i.e. if $\mu_{A(t)}$ is the
bracket flow starting at $\mu_{A_0},$ we will study
$\tfrac{A(t)}{\|A(t)\|}.$ We use the positive, non-increasing
function obtained in Section \ref{nue} to determine which limits correspond to
flat manifolds. Before stating the theorem of convergence, we demonstrate the following technical lemma. From now on, let $B(t)
= \tfrac{A(t)}{\|A(t)\|}.$
\begin{lemma}\label{derivadas}
The following evolution equations along the normalized flow by the bracket norm hold:
\begin{itemize}
\item[(i)] $\tfrac{d}{dt}\tr(B)=\tfrac{1}{2}\|A\|^2\|[B,B^t]\|^2 \tr(B),$
\item[(ii)] $\tfrac{d}{dt}\tr(B^2)=\|A\|^2\|[B,B^t]\|^2\tr(B^2).$
\end{itemize}
\end{lemma}
\begin{proof}
To prove (i), we use (\ref{a'}) and (\ref{normA}). Part (ii) follows from (\ref{tra2}) and (\ref{normA}).
\end{proof}
\begin{theorem}\label{convnorsol}
For any sequence $t_k \rightarrow \infty,$ there exists a subsequence of $(G_{\mu_{B(t_k)}}, g_{\mu_{B(t_k)}})$
converging in the pointed topology to an algebraic soliton
$(G_{\mu_{B_{\infty}}}, g_{\mu_{B_{\infty}}}).$ Moreover, the
following conditions are equivalent:
\begin{itemize}
\item[(i)] $\Spec(A_0) \subseteq i \RR.$
\item[(ii)] $(G_{\mu_{B_{\infty}}}, g_{\mu_{B_{\infty}}})$ is flat.
\end{itemize}
\end{theorem}
\begin{proof}
As $\|B(t)\| = 1,$ every sequence has a convergent subsequence, i.e., $B(t_k)$ converge to $B_{\infty},$ which is an algebraic soliton (see Corollary \ref{notafunc}).
By using (\ref{fc}), we have that
\begin{equation}\label{normspe}
\Spec(B(t_k)) = \Spec(\tfrac{A(t_k)}{\|A(t_k)\|}) = \tfrac{a(t_k)}{\|A(t_k)\|} \Spec(A_0).
\end{equation}
If $\Spec(A_0) \subseteq i \RR,$ then $\tr(B(0)^2) = \tr \left(\tfrac{{A_0}^2}{\|A_0\|^2}\right) < 0,$ and so by Proposition \ref{derivadas} (ii), we have that $\tr(B(t)^2) < 0$ for all $t,$ and $\tr(B(t)^2)$ is a decreasing function. It follows that $\tr({B_{\infty}}^2) < 0$ and then $B_{\infty}$ is normal, as $B_{\infty}$ is an algebraic soliton (see Proposition \ref{solmuA}). So, by (\ref{normspe}), we have that $\Spec(B_{\infty}) \subseteq i \RR$ and so $B_{\infty}$ is a skew-symmetric matrix.
Conversely, if $B_{\infty}$ is skew-symmetric, then $\Spec(B_{\infty}) \subseteq i \RR$, so, by using (\ref{normspe}), we have
that $\Spec(A_0) \subseteq i \RR.$
\end{proof}
Here again, we wonder ourselves what happens with the $\omega$-limit of $\tfrac{A(t)}{\|A(t)\|}.$ Recall that in Section \ref{punlim} we saw that if $\tr(A_0) = 0,$ then the $\omega$-limit of $\tfrac{A(t)}{\|A(t)\|}$ is a single point. In the following proposition we analyze the case $\tr(A_0) \neq 0.$
\begin{proposition}\label{omelim}
If $\tr(A_0)\neq 0$ and $B(t_k) \rightarrow B_{\infty},$ for some sequence $t_k \rightarrow \infty,$ then the $\omega$-limit of $\tfrac{A(t)}{\|A(t)\|}$ is contained in $O(n).B_{\infty}.$
\end{proposition}
\begin{proof}
Let $A_0$ be such that $\tr(A_0) \neq 0$ and we suppose that $B(t_k) \rightarrow B_{\infty}^1$ and $B(s_l) \rightarrow B_{\infty}^2.$ We want to see that $B_{\infty}^1$ and $B_{\infty}^2$ are conjugate by an orthogonal matrix.
\begin{itemize}
  \item If $\tr(A_0)<0,$ then $\tr(B(0))<0$ and by Proposition \ref{derivadas} (i), $\tr(B(t))<0,$ and therefore $\tr(B(t))$ is a decreasing function and it follows that $\tr(B(t)) < \tr(B(0)),$  for all $t.$
  \item If $\tr(A_0)>0,$ then $\tr(B(0))>0$ and by Proposition \ref{derivadas} (i), $\tr(B(t))>0,$ and therefore $\tr(B(t))$ is an increasing function and it follows that $\tr(B(0)) < \tr(B(t)),$ for all $t.$
\end{itemize}
Then, $\tr(B_{\infty}^1) \neq 0$ and $\tr(B_{\infty}^2) \neq 0.$ Furthermore, the function $\tr(B(t))$ is either increasing or decreasing. So, $\tr(B_{\infty}^1) = \tr(B_{\infty}^2).$ From this and (\ref{fc}) it follows that
$$
\lim_ {k\rightarrow \infty} \tfrac{a(t_k)}{\|A(t_k)\|} \tr(A_0)=\tr(B_{\infty}^1)= \tr(B_{\infty}^2)= \lim_ {l\rightarrow \infty} \tfrac{a(s_l)}{\|A(s_l)\|} \tr(A_0).
$$ and
$$
\Spec(B_{\infty}^1) = \lim_ {k\rightarrow \infty} \tfrac{a(t_k)}{\|A(t_k)\|} \Spec(A_0) = \lim_ {l\rightarrow \infty} \tfrac{a(s_l)}{\|A(s_l)\|} \Spec(A_0) = \Spec(B_{\infty}^2).
$$
Finally, we observe that $B_{\infty}^1$ and $B_{\infty}^2$ are normal matrices, since $\mu_{B_{\infty}^1}$ and $\mu_{B_{\infty}^2}$ are algebraic solitons (see Corollary \ref{notafunc}), and so, $B_{\infty}^1$ and $B_{\infty}^2$ are normal or nilpotent (see Proposition \ref{solmuA}). As $\tr(B_{\infty}^1) \neq 0$ and $\tr(B_{\infty}^2) \neq 0,$ they are not nilpotent matrices. Then, we have two normal matrices with the same spectrum, from which it follows that they are conjugate by an orthogonal matrix (see \cite{Hffmn}).
\end{proof}
\section{Negative Curvature}\label{curv}
In this section, we are interested in how the curvature evolves along the Ricci flow. We define the sectional curvature $K$ of $(\ggo,\ip),$ a Lie algebra endowed with an inner product, as the sectional curvature of $(G,g),$ where $G$ is the simply connected Lie group with Lie algebra $\ggo$ and $g$ is the left-invariant metric in $G$ such that $g(0)=\ip.$ In the case of $(\RR^{n+1},\mu_A,\ip),$ we simply denote it by $K_A.$ We say that a Riemannian manifold has negative curvature, and denote it by $K<0,$ if all sectional curvatures are strictly negative.

Next, we enunciate two results proved by Heintze in \cite{Hntz}. Theorems \ref{solk} and \ref{adneg} give necessary and sufficient conditions for certain solvable Lie algebras with an inner product to have negative curvature and for a solvable Lie algebra to admit an inner product with negative curvature, respectively.
\begin{theorem}\cite[Theorem 1]{Hntz}\label{solk}
 Let $(\ggo,\ip)$ be a solvable Lie algebra with an inner product such that the derived algebra is abelian (i.e., $\ggo'=[\ggo,\ggo]$ abelian). Then  $K<0$ if and only if the following conditions hold:
\begin{itemize}
  \item[(A)]$\dim \ggo'=\dim \ggo-1.$
  \item[(B)] There exists a unit vector $A_0 \in \ggo,$ orthogonal to $\ggo',$ such that $D_0:\ggo'\rightarrow\ggo'$ is positive definite, where $D_0$ is the symmetric part of $\ad_{A_0}|_{\ggo'}:\ggo'\rightarrow\ggo'.$
  \item[(C)] If $S_0$ is the skew-symmetric part of $\ad_{A_0}|_{\ggo'}$, then $D_0^2 + [D_0,S_0]|_{\ggo'}$ is also positive definite.
\end{itemize}
\end{theorem}
\begin{remark}\label{notathm2}
We observe that in the case of $\mu_A,$ the assumption that the
derived algebra is abelian is always true. Furthermore, $K_A<0$ if
and only if conditions (A) - (C) hold. If in addition $A$ is normal and invertible, then $K_A<0$ if and only if (B)
holds, since condition (A) is satisfied as $A$ is invertible and condition (C) follows from (B).
\end{remark}
\begin{theorem}\cite[Theorem 3]{Hntz}\label{adneg}
Let $\ggo$ be a solvable Lie algebra. Then the following conditions are equivalent:
\begin{itemize}
  \item[(i)]$\ggo$ admits an inner product with negative curvature.
  \item[(ii)] $\dim \ggo'=\dim \ggo-1$ and there exists $A_0 \in \ggo$ such that $\Re (\Spec (\ad_{A_0}|_{\ggo'}))>0.$
\end{itemize}
\end{theorem}
\begin{remark}\label{notathm3}
Note that if $A$ is invertible, then $G_{\mu_A}$ admits a left-invariant metric with $K<0$ if and only if either $\Re (\Spec (A))>0$ or $\Re (\Spec (A))<0.$
\end{remark}
\begin{theorem}\label{to2}
 Let $G_{\mu_{A_0}}$ be a solvable Lie group that admits a left-invariant metric with negative curvature. If $\mu_{A(t)}$ is the bracket flow starting at $\mu_{A_0},$ then there exists $s_0 \in \RR$ such that $K_{A(t)}<0,$ for all $t \geq s_0.$
\end{theorem}
\begin{proof}
It is sufficient to prove that the theorem holds for $B(t)=\tfrac{A(t)}{\|A(t)\|},$ i.e. there exists $t_0 \in \RR$ such that $K_{B(t)}<0,$ for all $t \geq t_0.$ Indeed, for each $t,$ $\mu_{A(t)}$ and $\mu_{B(t)}$ differ only by scaling.

By assumption, $G_{\mu_{A_0}}$ admits a left-invariant metric with negative curvature, then by using Remark \ref{notathm3} we have that either $\Re(\Spec(A_0)) > 0$ or $\Re(\Spec(A_0)) < 0.$

Assume that, after passing to a subsequence, $B(t_k)$ converges to $B_{\infty},$ as $k \rightarrow \infty.$ Then, arguing as in Proposition
\ref{omelim}, we have that $B_{\infty}$ is normal and
$$
\Spec(B_{\infty}) = \alpha \Spec(A_0), \quad \alpha \in \RR,
\alpha \neq 0.
 $$ so, either $\Re (\Spec (B_{\infty}))>0$ or $\Re (\Spec (B_{\infty}))<0.$ Then $S(B_{\infty})$ is either positive or negative definite. It follows by Remark \ref{notathm2} that $K_{B_{\infty}}<0.$ Thus, there exists $L \in \NN$ such that $K_{B(t_k)}<0,$ for all $k \geq L.$

Finally, there must exist $t_0$ such that $K_{B(t)}<0,$ for all $t \geq t_0,$ otherwise we would be able to extract a convergent subsequence $B(t_k),$ whose sectional curvatures are not strictly negative, and this contradicts the previous paragraph. 
\end{proof}
We now show that the above theorem is not longer valid in the general solvable case.
\begin{example}\label{ejsol} We consider $(\mu_{\lambda,\alpha}, \ip)$ defined as follows:
$$
\begin{array}{cc}
  \mu_{\lambda,\alpha}(e_0,e_i)=\alpha \left(
              \begin{array}{ccc}
                \lambda &  &  \\
                 & 1-\lambda &  \\
                 &  & 1 \\
              \end{array}
            \right)e_i,
 & \mu_{\lambda,\alpha}(e_1,e_2)=e_3,
\end{array}
$$ and $\ip$ the inner product for which $\{e_0,e_1,e_2,e_3\}$ is an orthonormal basis. By \cite[Theorem 4.8]{riccisol}, we know that $(\mu_{\lambda,\alpha}, \ip)$ is an algebraic soliton if and only if $\alpha=\tfrac{\sqrt{3}}{\sqrt{2(\lambda^2+(1-\lambda)^2+1)}}.$ We consider the $2$-dimensional plane $\pi = \la e_1,e_3\ra$ and we compute its sectional curvature:
$$
K(e_1,e_3)=\|U(e_1,e_3)\|^2-\langle U(e_1,e_1), U(e_3,e_3) \rangle = \tfrac{1}{4} - \tfrac{3 \lambda}{\lambda^2+(1-\lambda)^2+1}.
$$ So,
$$K(e_1,e_3) \geq 0 \quad \Leftrightarrow \quad \tfrac{1}{4} - \tfrac{3 \lambda}{\lambda^2+(1-\lambda)^2+1} \geq 0 \quad \Leftrightarrow \quad \lambda \leq 2-\sqrt{3} \quad \mbox{ \'o} \quad \lambda \geq 2+\sqrt{3}.$$
We observe that if $0 < \lambda \leq 2-\sqrt{3},$ then $0 < 1-\lambda,$ and so $\ad (e_0)$ is a matrix such that $\Re (\Spec( \ad(e_0)))>0.$ Then, Theorem \ref{adneg} said that if $0 < \lambda \leq 2-\sqrt{3},$ then $(\mu_{\lambda,\alpha}, \ip),$ with $\alpha=\tfrac{\sqrt{3}}{\sqrt{2(\lambda^2+(1-\lambda)^2+1)}},$ admits an inner product with negative curvature. On the other hand, since $(\mu_{\lambda,\alpha}, \ip)$ is an algebraic soliton, if $\mu(t)$ is the bracket flow starting at $\mu_{\lambda,\alpha},$ then $(G_{\mu(t)},g_{\mu(t)})$ has planes with curvature bigger than or equal to zero. 
\end{example}
The next question is what happens with the Ricci flow when we start with a metric whose sectional curvatures are all negative. First, we will introduce a theorem proved by Heintze in \cite{Hntz}.

Let $(\ggo,\ip)$ be a solvable Lie algebra with an inner product such that (A) - (C) of the Theorem \ref{solk} hold. Then, we have a orthogonal decomposition $\ggo={A_0 + [\ggo,\ggo]}.$ For $\alpha >0,$ let $(\ggo_{\alpha},\ip)$ be the Lie algebra with the same inner product that $(\ggo, \ip)$ but with the following modification in the Lie bracket
$$
[A_0,X]_{\alpha}:= \alpha [A_0,X], \mbox{ para todo } X \in \ggo'=\ggo_{\alpha}'.
$$
\begin{theorem}\cite[Theorem 2]{Hntz}\label{solub}
 Let $(\ggo,\ip)$ be a solvable Lie algebra with an inner product and assume that (A)-(C) hold. Then there exists $\alpha_0 >0$ such that $(\ggo_{\alpha},\ip)$ has negative curvature for all $\alpha \geq \alpha_0.$
\end{theorem}
We return to Example \ref{ejsol}. Let $\lambda$ be fixed and  we consider the bracket flow $\mu(t)$ starting at $\mu_{\alpha,\lambda}.$ Then $\mu(t)$ is given by
$$
\mu(t)(e_0,e_i)=\alpha(t)\left(
                           \begin{array}{ccc}
                             \lambda &  &  \\
                              & 1-\lambda &  \\
                              &  & 1 \\
                           \end{array}
                         \right)e_i, \quad \mu(t)(e_1,e_2)=h(t)e_3,
$$ with $\alpha=\alpha(t)$ and $h=h(t)$ that satisfy the following differential equations:
$$
\left\{
  \begin{array}{ll}
    \alpha'=-c_{\lambda}\alpha^3, & \alpha(0)=\alpha, \\
    h'=-\tfrac{3}{2}h^3, & h(0)=1,
  \end{array}
\right.
$$ where $c_{\lambda}=(\lambda^2+(1-\lambda)^2+1).$ Furthermore, solving the equations we obtain that $\alpha(t)=\tfrac{1}{\sqrt{2c_{\lambda}t+\alpha^{-2}}}$ and $h(t)=\tfrac{1}{\sqrt{3t+1}}.$ Clearly, in this case, the bracket flow converge to a flat metric, but for fixed $t,$ we have that
$$
K(e_1,e_3) = \tfrac{h^2}{4} - \lambda \alpha^2 = \tfrac{1}{4(3t+1)} - \tfrac{\lambda}{2c_{\lambda}t+\alpha_0^{-2}}.
$$ Then,
$$
K(e_1,e_3) \geq 0 \quad \Leftrightarrow \quad \tfrac{1}{4(3t+1)} \geq \tfrac{\lambda}{2c_{\lambda}t+\alpha_0^{-2}} \quad \Leftrightarrow \quad (2 c_{\lambda} - 12 \lambda)t \geq 4 \lambda - \alpha_0^{-2}.
$$ Further, $2 c_{\lambda} - 12 \lambda = 4 ((\lambda - 2)^2-3).$ So, if $0 < \lambda \leq 2-\sqrt{3},$ there exists $t_0$ such that $K(e_1,e_3) \geq 0, \forall t \geq t_0.$

Let $\lambda$ be such that $0 < \lambda \leq 2-\sqrt{3},$ and we consider $\mu_{\alpha,\lambda}, \alpha \in \RR_{>0}.$ Then $(\mu_{\alpha,\lambda},\ip)$ is a solvable Lie algebra with an inner product that satisfies (A) - (C). By Theorem \ref{solub}, we know that there exists $\alpha_0 >0$ such that $((\mu_{\alpha,\lambda})_{\alpha_0},\ip)$ has negative curvature. Then, $(\mu_{\alpha \alpha_0,\lambda},\ip)$ has a negative curvature.
On the other hand, we know that if $\mu(t)$ is the bracket flow starting at $\mu_{\alpha \alpha_0,\lambda}$ there exists $t_0$ such that $\forall t \geq t_0,$ $(G_{\mu(t)},g_{\mu(t)})$ has planes with curvature bigger than or equal to zero.

\end{document}